\documentclass[]{article}
\usepackage[
    colorlinks=true,%
    breaklinks,
    linkcolor=black,%
    anchorcolor=black,%
    citecolor=black,%
    filecolor=black,%
    menucolor=black,%
    pagecolor=black,%
    urlcolor=black,%
    hyperfigures,
    bookmarksnumbered%
]{hyperref}
\usepackage{natbib}
\usepackage{amsmath,amssymb,latexsym,hyperref,url,graphicx,amsthm}
\usepackage[letterpaper,margin=3.5cm]{geometry}
\usepackage{afterpage}

\usepackage[usenames,dvipsnames]{xcolor} 

\newtheorem{definition}{Definition}

\newtheorem{lemma}{Lemma}
\newtheorem{notation}{Notation}

\newtheorem{remark}{Remark}

\newtheorem{assumption}{Assumption}

\title{An Analysis of Pseudo-Goodwin Cycles \\ in a Wage-Led Minsky Model}
\author{Johannes Buchner}

\begin{document}

\maketitle

\begin{abstract}
The goal of these notes is to make mathematically more precise some of the ideas presented in~\citet{Stockhammer2014}. At first the title seems like a contradiction to have a wage-led model and still find goodwin cycles in it, but the point we try to make in the paper is that those are only `pseudo-goodwin' cycles, and not real goodwin cycles.
\end{abstract}

\section{Acknowledgements}

This paper is based on discussions with Charlie Brummit, who I am deeply indebted, especially for performing the simulations. I have coined the name "pseudo goodwin cycles" in reference to "pseudo trajectories", which play an important role in numerical mathematics of dynamical systems. I would like to thank Engelbert Stockhammer, who introduced me to the topic.

\section{Definitions}

\subsection{Variables}
\begin{enumerate}
\item Time $t \in \mathbb{R}$ is continuous.
\item \emph{Output} of the economy is denoted by $y(t) \geq 0$.
\item A portion $w(t)$ of output $y(t)$ is distributed to \emph{wages}. [The quantity $w(t)$ will also be called the \emph{distribution}.]. The remainder, $y(t) - w(t)$, is \emph{profits}. 
\item Financial fragility $f(t)$ is the ratio of debt-to-income of firms. 
\end{enumerate}

\begin{notation}
Derivatives with respect to time, $d/dt$, will denoted by an overdot, such as $\dot y(t)  \equiv d y(t) / d t$. For short, we will just write $ \dot y$.
\end{notation}
\begin{notation}
We will denote the nonnegative real numbers by $\mathbb{R}_+ \equiv \{x \in \mathbb{R} : x \geq 0\}$.
\end{notation}

\subsection{Differential equations}

We will consider a system of three coupled ordinary differential equations (ODEs) of the form
\begin{subequations}
\begin{align}
\dot y &= F(y,w,f)\\
\dot w &= G(y,w,f)\\
\dot f &= H(y,w,f),
\end{align}
\label{eq:general_diff_eq}
\end{subequations}
where $F,G,H : \mathbb{R}_+^3 \to \mathbb{R}$ are ``sufficiently nice'' functions to be specified later.
Note that System~\eqref{eq:general_diff_eq} is \emph{autonomous} [meaning that $f, g, h$ do not depend explicitly on time $t$; instead they only depend on the variables $y(t), w(t), f(t)$].

\section{A simplified version of the Goodwin model}
The Goodwin model~\citep{Goodwin1967} posits that a business cycle occurs from ``the interaction between a profit-led demand function and a reserve army distribution function''.

\begin{definition}
A \emph{profit-led demand function} means that $\partial \dot y / \partial w \equiv \partial F(y,w,f) / \partial w  < 0$ for all sufficiently large $t$. 
\end{definition}
\begin{remark} The economic reasoning is that an increase in the wage share $w(t)$ is associated with a decline in demand because investment is driven by the profit share $y(t) - w(t)$. 
\end{remark}

\subsection{Reserve army effect}

\begin{assumption}[Output varies positively with employment]
Higher levels of output $y(t)$ are associated with higher levels of employment.
\label{assumption:output_employment}
\end{assumption}

\begin{assumption}[Reserve army assumption]
``[A]s unemployment increases the bargaining power of labour is diminished, leading to a fall in the wage rate'' $w(t)$. And vice versa: as unemployment falls, the bargaining power of labour increases, leading to a rise in the wage rate $w(t)$. 
\label{assumption:reserve_army}
\end{assumption}

Combining Assumptions~\ref{assumption:output_employment} and~\ref{assumption:reserve_army} gives
\begin{assumption}[Large output $y(t)$ is associated with rising wage share $w(t)$]
When output $y(t)$ is large, employment is large by Assumption~\ref{assumption:output_employment}, so by Assumption~\ref{assumption:reserve_army} the wage share $w(t)$ rises. 
\label{assumption:output_wage}
\end{assumption}

Now we make Assumption~\ref{assumption:output_wage} mathematically precise:
\begin{definition}[Reserve army effect]
When $y(t)$ is sufficiently large, we have $\dot w > 0$. That is, there exists a constant $\kappa > 0$ such that $y(t) > \kappa \implies \dot w > 0$.
\label{def:reserve_army}
\end{definition}

\subsection{Profit squeeze theory of accumulation}


\begin{definition}[Profit squeeze]
For sufficiently large wage share $w(t)$, output $y(t)$ falls. That is, there exists a constant $\omega > 0$ such that $w(t) > \omega \implies \dot y < 0$.
\label{def:profit_squeeze}
\end{definition} 

This definition is motivated by the following assumptions (footnote 1 on page 2 of~\citet{Stockhammer2014}):
\begin{assumption}
``[W]orkers consume all wages and capitalists invest all profits, the marginal productivity of labour is constant and the capital-output ratio is fixed. Given these assumptions, a rise in the real wage will lead to a reduction in the growth of profits and - since investment is equal to profit - a reduction in the rate of output growth.''
\end{assumption}

\subsection{A simplified version of the Goodwin model}
Let 
\begin{subequations}
\begin{align}
\dot y &= F_\text{Goodwin}(y,w,f) :=  y (1-w) \label{eq:output_Goodwin}\\
\dot w &= G_\text{Goodwin}(y,w,f) := w(-c+ry), \label{eq:wage_Goodwin}
\end{align}
\label{eq:Goodwin_model}
\end{subequations}
where $c$ and $r$ are positive constants. 

Note that Eq.~\eqref{eq:wage_Goodwin} satisfies the reserve army effect (Definition~\ref{def:reserve_army}) with $\kappa := r/c$, and Eq.~\eqref{eq:output_Goodwin} satisfies the profit squeeze effect (Definition~\ref{def:profit_squeeze}) with $\omega := 1$. System~\eqref{eq:Goodwin_model} makes several assumptions that simplify the original Goodwin model~\cite{Goodwin1967}:
\begin{assumption}
``[I]n the original model, it is assumed that labour productivity and the labour force grow at a steady exogenous rate while in our formulation we assume instead that labour productivity is constant. Our model thus generates cycles in output around the steady state, while the original Goodwin model is a growth model.''
\end{assumption}

Figure~\ref{fig:Goodwin_orbit} shows one sample orbit of the Goodwin model.

\begin{figure}[htb]
\begin{center}
\includegraphics{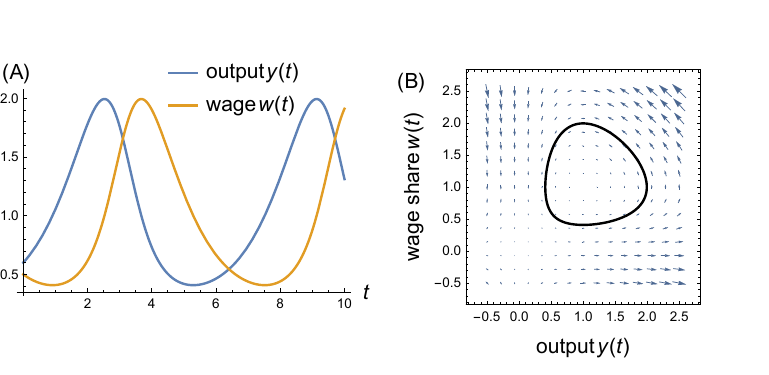}
\caption{Sample orbit of the Goodwin model~\eqref{eq:Goodwin_model} for $r=c=1$ and initial condition $y(0)=0.6,w(0)=0.5$. The time series in panel (A) show that peaks in output $y(t)$ precede peaks in the wage rate $w(t)$. Panel (B) shows the vector field given by System~\eqref{eq:Goodwin_model} using blue arrows, and the orbit (shown in black) moves counterclockwise.}
\label{fig:Goodwin_orbit}
\end{center}
\end{figure}

Figure~\ref{fig:Goodwin_variables} shows the feedbacks between the variables $y(t)$ and $w(t)$ in this Goodwin model.

\begin{figure}[htb]
\begin{center}
\includegraphics{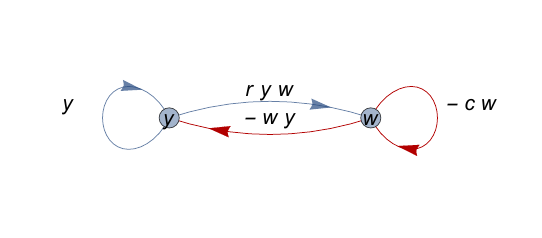}
\caption{Relationship among variables in the Goodwin model [System~\eqref{eq:Goodwin_model}]. Blue edges denote positive feedbacks; red edges denote negative feedbacks. Edge labels are the terms in the ODEs in System~\eqref{eq:Goodwin_model}.}
\label{fig:Goodwin_variables}
\end{center}
\end{figure}

\begin{definition}[Goodwin cycle]
A \emph{Goodwin cycle} is a counterclockwise closed orbit in $\left (y(t), w(t) \right )$ such that $\dot y$ depends on $w$ and $\dot w$ depends on $y$.
\end{definition}

\section{Minsky models}

\subsection{A simple Minsky model}
\begin{assumption}[Increasing financial fragility during booms]
During ``boom'' times, firms become more optimistic, so they take on more debt relative to their cash flows. Banks also become more willing to lend during booms.
\end{assumption}

\begin{definition}[Minsky financial fragility]
Financial fragility $f(t)$ exhibits the \emph{Minsky effect} if 
the rate of change of $f(t)$ increases with output $y(t)$, i.e., if $\partial \dot f / \partial y \equiv \partial H / \partial y > 0$.
\label{def:Minsky_f}
\end{definition}

\begin{definition}[Inverse relationship between output growth and financial fragility]
The rate of change of output decreases with financial fragility, i.e., 
$\partial \dot y / \partial f \equiv \partial H \left (y,w,f \right )/\partial f < 0$ for all $y(t),w(t),f(t) \geq 0$.
\label{def:output_growth_f}
\end{definition}

Let 
\begin{subequations}
\begin{align}
\dot y &= F_\text{Minsky}(y,w,f) :=  y (1-f) \label{eq:output_Minsky}\\
\dot f &= H_\text{Minsky}(y,w,f) := f(-1+p y), \label{eq:f_Minsky}
\end{align}
\label{eq:Minsky_model}
\end{subequations}
where $p$ is a positive constant. Note that Eq.~\eqref{eq:f_Minsky} satisfies the Minsky effect (Definition~\ref{def:Minsky_f}) provided that $f(t) \geq 0$, and note that Eq.~\eqref{eq:output_Minsky} satisfies the inverse relationship between output growth $\dot y$ and financial fragility $f(t)$ (Definition~\ref{def:output_growth_f}) provided that $f(t) \geq 0$.

System~\eqref{def:Minsky_f} exhibits countercyclical closed orbits much like the Goodwin model~\eqref{eq:Goodwin_model} does, as illustrated in Figure 4 of~\citet{Stockhammer2014}.

\begin{remark}
If we ignore the economic interpretations, then the Minsky model~\eqref{eq:Minsky_model} is a special case of the Goodwin model of the Goodwin model~\eqref{eq:Goodwin_model} with the replacement of variables $w \to f, c \to 1, r \to p$, so it is straightforward to see why closed orbits appear in both models.
\end{remark}

Figure~\ref{fig:Minsky_orbit} shows one sample orbit of the Minsky model.

\begin{figure}[htb]
\begin{center}
\includegraphics{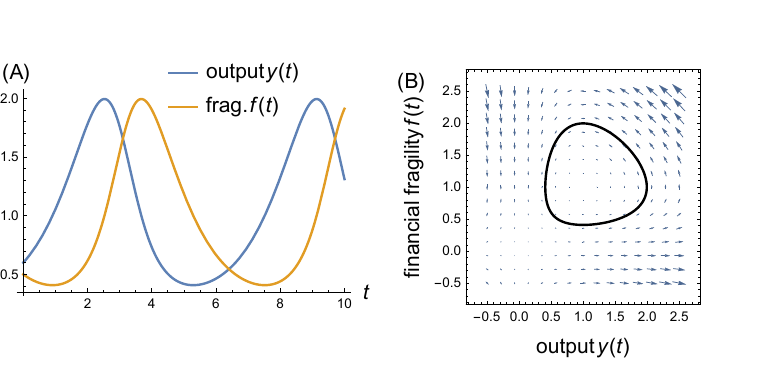}
\caption{Sample orbit of the Minsky model~\eqref{eq:Minsky_model} for $p=1$ and initial condition $y(0)=0.6,f(0)=0.5$. The time series in panel (A) show that peaks in output $y(t)$ precede peaks in the wage rate $w(t)$. Panel (B) shows the vector field given by System~\eqref{eq:Minsky_model} using blue arrows, and the orbit (shown in black) moves counterclockwise.}
\label{fig:Minsky_orbit}
\end{center}
\end{figure}

Figure~\ref{fig:Minsky_variables} shows the feedbacks between the variables $y(t)$ and $w(t)$ in this Minsky model.

\begin{figure}[htb]
\begin{center}
\includegraphics{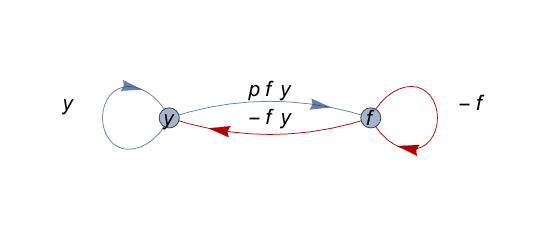}
\caption{Relationship among variables in the Minsky model [System~\eqref{eq:Minsky_model}]. Blue edges denote positive feedbacks; red edges denote negative feedbacks. Edge labels are the terms in the ODEs in System~\eqref{eq:Minsky_model}.}
\label{fig:Minsky_variables}
\end{center}
\end{figure}

\subsection{Minsky cycle with a reserve army effect}

We augment System~\eqref{eq:Minsky_model} with dynamics for the wage share that exhibits the reserve army effect:
\begin{align}
\dot w = G_\text{Minsky-reserve-army}(y,w,f) := w (-c+ry-w). \label{eq:wage_Minsky}
\end{align}

Note that the wage distribution $w(t)$ has no effect on demand. Said mathematically, $w(t)$ is \emph{enslaved} to $y(t)$; Eq.~\eqref{eq:wage_Minsky} is a \emph{slave equation}, while Eq.~\eqref{eq:output_Minsky} is called a \emph{master equation}. \citet{Stockhammer2014} also invoke the language of a \emph{scavenger} from biological models to describe the slave equation~\eqref{eq:wage_Minsky}.

Equation~\eqref{eq:wage_Minsky} is similar to the equation for $\dot w$ in the Goodwin model, Eq.~\eqref{eq:wage_Goodwin}, but with an extra $-w^2$ term ``to contain the rate of real wage increases and counteract the wage demands of workers at higher output''~\citep{Stockhammer2014}. 

\begin{definition}[Pseudo-Goodwin cycle]
A \emph{pseudo-Goodwin cycle} is a counterclockwise closed orbit in $\left (y(t), w(t) \right )$ such that $y(t)$ and $w(t)$ form a master--slave system.
\end{definition}

\begin{lemma}
A pseudo-Goodwin cycle necessarily has at least one more variable that is neither enslaved to $y(t)$ nor enslaves $y(t)$. 
\end{lemma}
\begin{proof}
Follows from the fact that the solution to a first-order ODE $\dot x = \mathcal{F}(x)$ cannot oscillate. 
\end{proof}

Figure~\ref{fig:Minsky_reserve_army_orbit} shows one sample orbit of the Minsky model with a reserve-army effect model. Note that the $-w^2$ term in Eq.~\eqref{eq:wage_Minsky} damps the wages to zero in this example. 

\begin{figure}[htb]
\begin{center}
\includegraphics{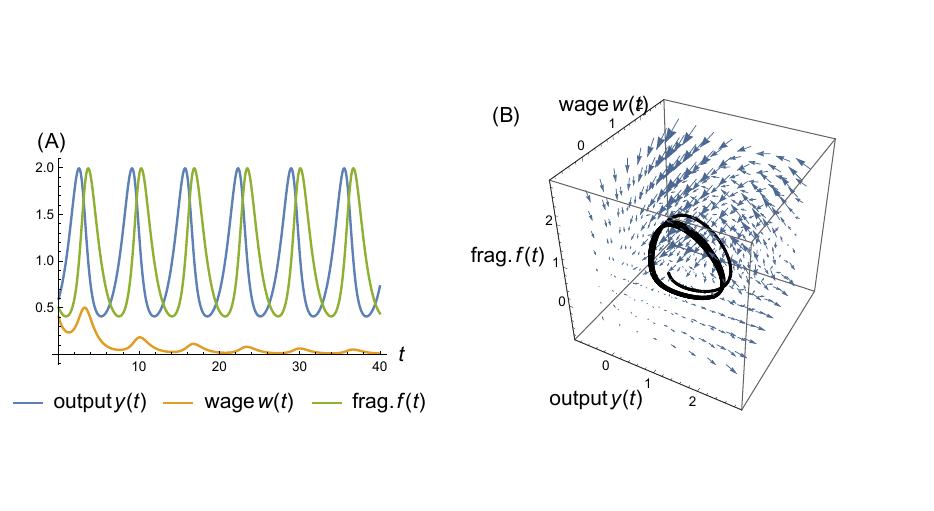}
\includegraphics{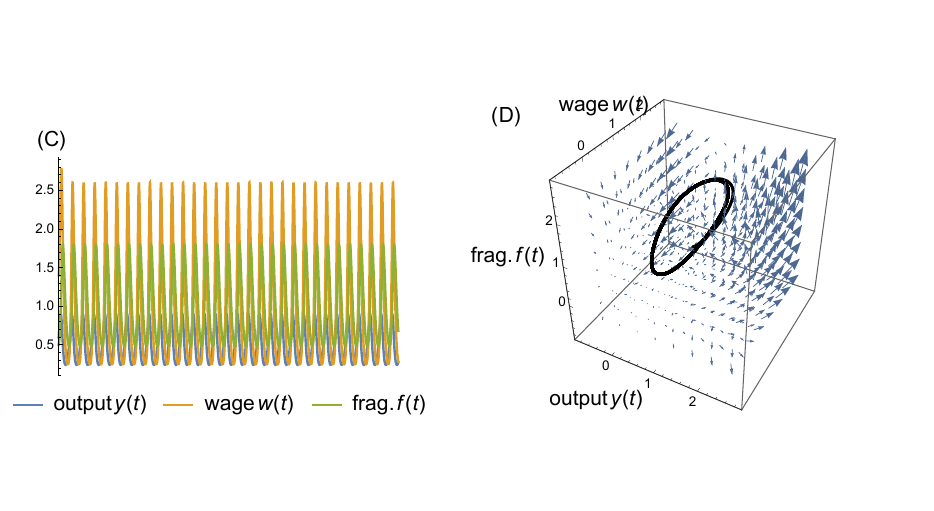}
\caption{Sample orbit of the Minsky model with a reserve army effect [System~\eqref{eq:Minsky_model} augmented with Eq.~\eqref{eq:wage_Minsky}] for $p=r=c=1$ and $40$ time steps (top row) and the parameters in~\citep{Stockhammer2014} ($p=2,r=5,c=3/2$, $200$ time steps, bottom row) and initial condition $y(0)=0.6,f(0)=0.5, w(0) = 0.4$. For the first set of parameters (top row), the wages $w(t)$ damp to zero, whereas for the second set of parameters (bottom row) the wage share $w(t)$ do not damp to zero (as in the example in~\citet{Stockhammer2014}).} 
\label{fig:Minsky_reserve_army_orbit}
\end{center}
\end{figure}

Figure~\ref{fig:Minsky_reserve_army_variables} shows the feedbacks between the variables $y(t), w(t), f(t)$ in this Minsky model with a reserve army effect.

\begin{figure}[htb]
\begin{center}
\includegraphics{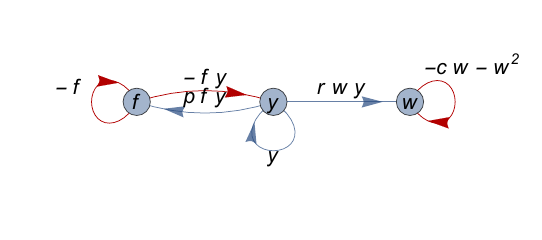}
\caption{Relationship among variables in the Minsky model with a reserve army effect [System~\eqref{eq:Minsky_model} augmented with Eq.~\eqref{eq:wage_Minsky}]. Blue edges denote positive feedbacks; red edges denote negative feedbacks. Edge labels are the terms in the ODEs in System~\eqref{eq:Minsky_model} together with Eq.~\eqref{eq:wage_Minsky}. Note that $w(t)$ is enslaved by $y(t)$ and $f(t)$, i.e., there is an edge from $y$ to $w$ but no edge from $w$ to $y$ nor from $w$ to $f$.}
\label{fig:Minsky_reserve_army_variables}
\end{center}
\end{figure}

\subsection{A Minsky model with distribution function and a wage-led effect in the demand function}

Now we retain Eqs.~\eqref{eq:f_Minsky} and~\eqref{eq:wage_Minsky} in the previous model (i.e. the Minsky model with a reserve army effect), but we modify the equation for the flow of output $\dot y$ [Eq.~\eqref{eq:output_Minsky}] by including a positive feedback from wage share $w(t)$ to demand and hence to output $y(t)$, namely, a new term $s y(t) w(t)$, where $s$ is a positive constant. The resulting system is
\begin{subequations}
\begin{align}
\dot y &= F_\text{Minsky-reserve-army}(y,w,f) :=  y (1-f + sw) \label{eq:output_Minsky_distribution_wage-led}\\
\dot w &= G_\text{Minsky-reserve-army}(y,w,f) := w (-c+ry-w). \label{eq:wage_Minsky_distribution_wage-led} \\
\dot f &= H_\text{Minsky}(y,w,f) := f(-1+p y). \label{eq:f_Minsky_distribution_wage-led}
\end{align}
\label{eq:Minsky_distribution_wage}
\end{subequations}

Figure~\ref{fig:Minsky_model_reserve_army_effect_wage-led_demand_sample_orbit} shows one sample orbit of this model. Note that the $-w^2$ term in Eq.~\eqref{eq:wage_Minsky} damps the wages to zero in this example.

\begin{figure}[htb]
\begin{center}
\includegraphics{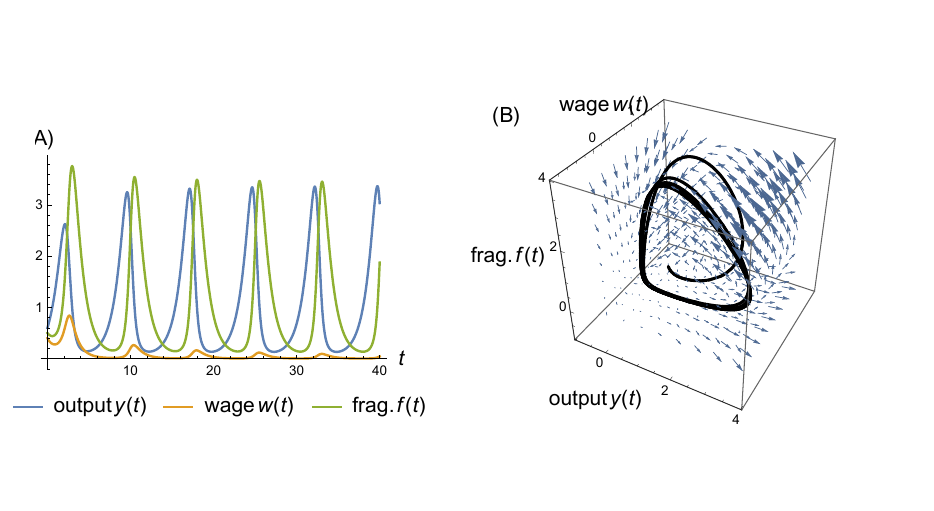}
\caption{Sample orbit of System~\eqref{eq:f_Minsky_distribution_wage-led} for $p=s=r=c=1$ and initial condition $y(0)=0.6,f(0)=0.5, w(0)=0.4$. The time series in panel (A) show that peaks in output $y(t)$ precede peaks in the wage rate $w(t)$ and that wages $w(t)$ oscillate yet damp to $0$. Panel (B) shows the vector field given by System~\eqref{eq:f_Minsky_distribution_wage-led} using blue arrows, and the orbit (shown in black) moves counterclockwise in $(y,w)$-space.}
\label{fig:Minsky_model_reserve_army_effect_wage-led_demand_sample_orbit}
\end{center}
\end{figure}

Figure~\ref{fig:Minsky_model_reserve_army_effect_wage-led_demand_variables} shows the feedbacks between the variables $y(t), w(t), f(t)$ in System~\eqref{eq:Minsky_distribution_wage}.

\begin{figure}[htb]
\begin{center}
\includegraphics{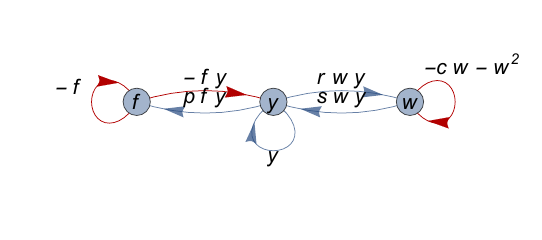}
\caption{Relationship among variables in the Minsky model with a reserve army effect [System~\eqref{eq:f_Minsky_distribution_wage-led}]. Blue edges denote positive feedbacks; red edges denote negative feedbacks. Edge labels are the terms in the ODEs in System~\eqref{eq:Minsky_distribution_wage}.}
\label{fig:Minsky_model_reserve_army_effect_wage-led_demand_variables}
\end{center}
\end{figure}

\subsection{Next steps}
Some possible ways to proceed:
\begin{enumerate}
\item Consider a fast--slow system with $|s| \ll 1$? \citet{Stockhammer2014} seem to have this case in mind because in their numerical example on page 12, they take $s := 0.02$.
\item \citet{Stockhammer2014} claim on page 12 that ``The cyclical behaviour of the model is still driven by the interaction between y and f (Figure 6b).'' I think this statement could be made more precise. It's certainly true when $s=0$, because then there is no feedback from $w$ to either $y$ or to $f$ (i.e., $w$ is enslaved by $y$ and $f$). If we let $s$ be negative but ``close'' to zero, then we have two possible ways for cycles to appear, in the interaction between $w$ and $y$ [as in the Goodwin model~\eqref{eq:Goodwin_model}] or in the interaction between $f$ and $y$ [as in the Minsky model~\eqref{eq:Minsky_model}]. Then we could ask whether there are stable limit cycles or centers or unstable orbits in that model.

\end{enumerate}


\newpage

\section{Hopf bifurcation in System~\eqref{eq:Minsky_distribution_wage} as $s$ crosses $0$}

Figure~\ref{fig:eigenvalues_Jacobian} shows the eigenvalues of the Jacobian matrix of System~\eqref{eq:Minsky_distribution_wage} evaluated at the unique interior fixed point $(y^*, w^*, f^*) = (1/p, r/p-c, r s/p - c s +1)$ as a function of $s$ from $-2$ (dark blue) to $2$ (light blue). Here, the parameters are the same as those used by~\citet[page 12]{Stockhammer2014}: $c=3/2, p = 2, r = 5$.

\begin{figure}[htbp]
\begin{center}
\includegraphics{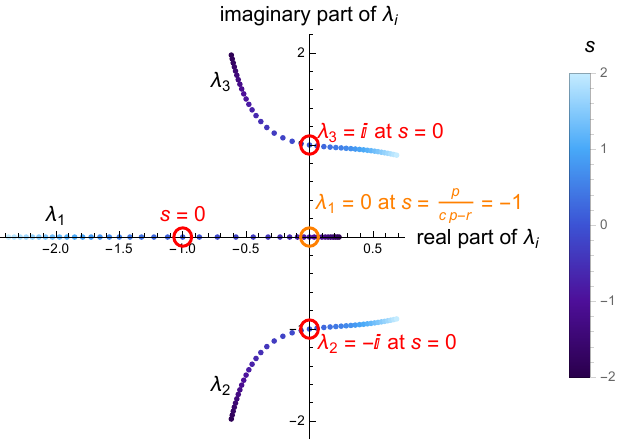}
\caption{Eigenvalues of the Jacobian matrix of System~\eqref{eq:Minsky_distribution_wage} evaluated at the unique interior fixed point $(y^*, w^*, f^*) = (1/p, r/p-c, r s/p - c s +1)$ as a function of $s$, which varies from $-2$ (dark blue) to $2$ (light blue) in this figure. Here, the parameters are the same as those used by~\citet[page 12]{Stockhammer2014}: $c=3/2, p = 2, r = 5$.
}
\label{fig:eigenvalues_Jacobian}
\end{center}
\end{figure}

At $s=0$, the first eigenvalue $\lambda_1 = -1$, and the other two eigenvalues $\lambda_2 = \bar \lambda_3$ cross the imaginary axis. This crossing of the imaginary axis by a pair of complex-conjugate eigenvalues suggests that a Hopf bifurcation occurs at $s=0$. Figure~\ref{fig:orbit_s=0} shows an orbit for $s=0$; the closed orbit of $y,f$ drives cyclic behavior in the enslaved variable $w$. 

For $s>0$, the complex-conjugate pair of eigenvalues have positive real part, i.e., $\Re \lambda_2 = \Re \lambda_3 > 0$, while the third eigenvalue $\lambda_3 < -1$. As a result, the dynamics (at least in the linear approximation near the interior fixed point) are unstable. 
This instability is also observed by~\citet[page 13]{Stockhammer2014}, who note that ``this cyclical motion is now unstable, so that outward spirals are generated in the phase space of all three pairs of variables''. Figure~\ref{fig:orbit_s=0p03} shows an orbit for $s=0.03$ that spirals outward. Eventually, the system blows up and escapes to $\infty$.

For $s < 0$ , the complex-conjugate pair of eigenvalues have negative real part, i.e., $\Re \lambda_2 = \Re \lambda_3 < 0$, and the other eigenvalue, $\lambda_1$, is negative if $s < p/(c p - r)$ and positive if $s>p/(c p - r)$ (see the orange text in Fig.~\ref{fig:eigenvalues_Jacobian}). Numerically, I find that all three variables converge to a fixed point (and oscillate on the way to the fixed point). In the appendix, some simulations and illustrations are presented.

\bibliographystyle{plainnat}

\bibliography{Business_cycle_bibliography}

\section{Appendix}

\newpage

\begin{figure}[htbp]
\begin{center}
\includegraphics[height=18.7cm]{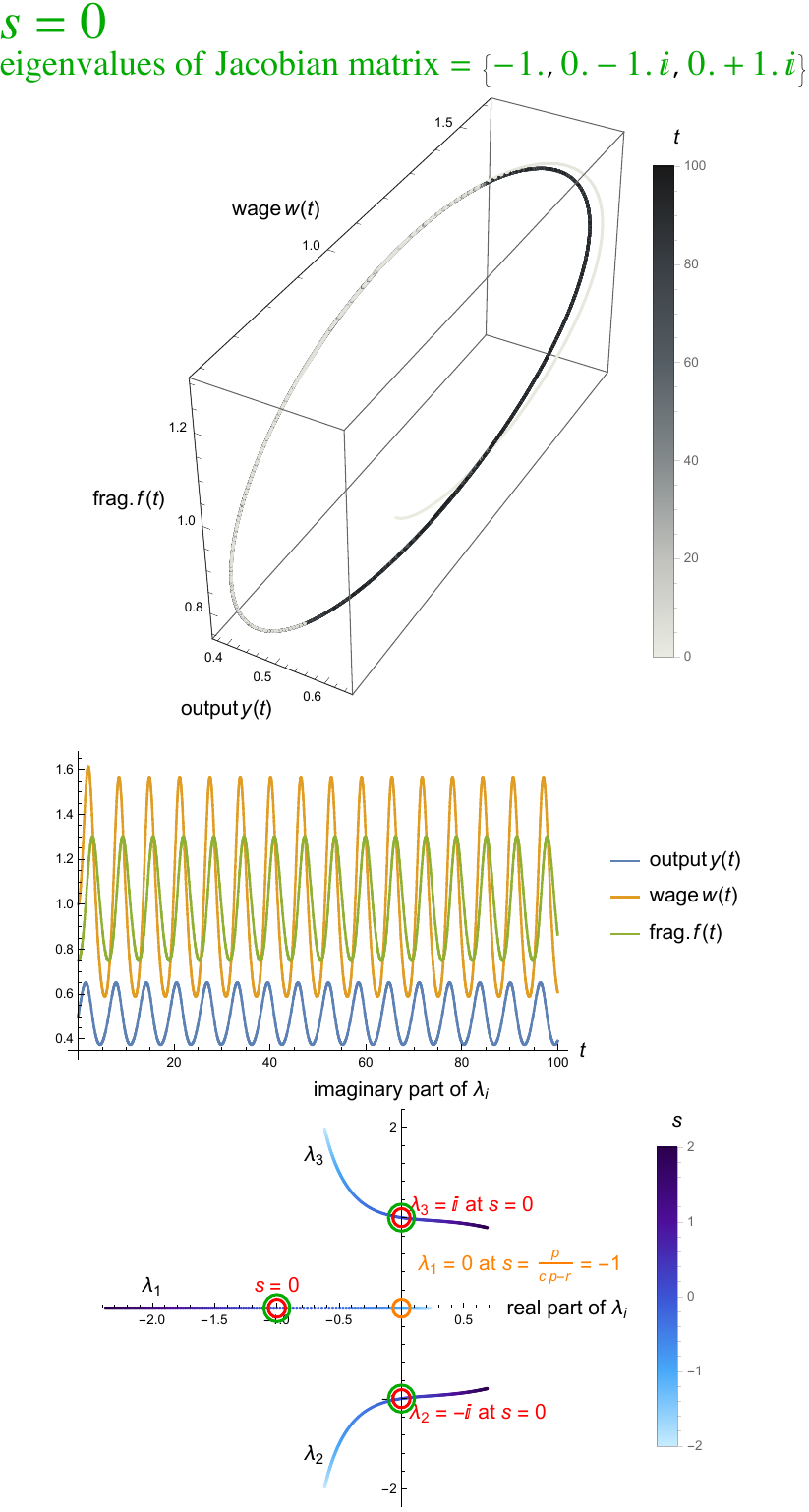}
\caption{Orbit for $s=0$. 
Time increases from $0$ (light gray) to $200$ (dark gray). The cycle in $y$ and $f$ drives cyclical behavior in the enslaved variable $w$. 
Here, the parameters are the same as those used by~\citet[page 12]{Stockhammer2014}: $c=3/2, p = 2, r = 5,y(0) = 1/2, f(0) = 3/4, w(0) = 1$. 
This case $s=0$ corresponds to the Minsky model with a reserve army effect [System~\eqref{eq:Minsky_model} augmented with Eq.~\eqref{eq:wage_Minsky}]; two example orbits of this model are shown in Fig.~\ref{fig:Minsky_reserve_army_orbit} (one of which has $w(t)$ being damped to zero over time). 
}
\label{fig:orbit_s=0}
\end{center}
\end{figure}

\begin{figure}[htbp]
\begin{center}
\includegraphics[height=20cm]{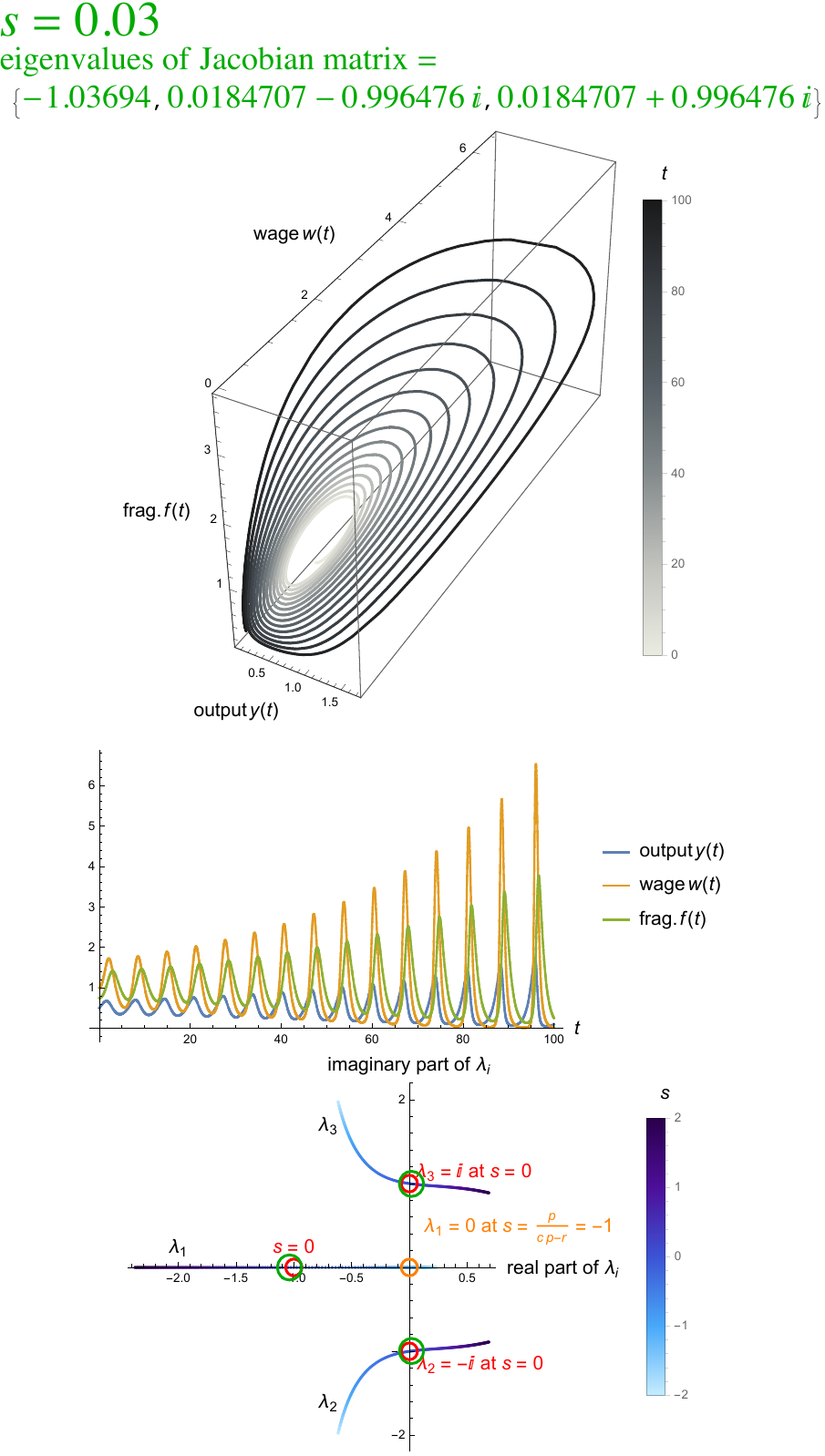}
\caption{Orbit for $s=0.03$. The dynamics are an unstable spiral. 
Time increases from $0$ (light gray) to $200$ (dark gray). 
Here, the parameters are the same as those used by~\citet[page 12]{Stockhammer2014}: $c=3/2, p = 2, r = 5,y(0) = 1/2, f(0) = 3/4, w(0) = 1$. 
}
\label{fig:orbit_s=0p03}
\end{center}
\end{figure}

\begin{figure}[htbp]
\begin{center}
\includegraphics[height=20cm]{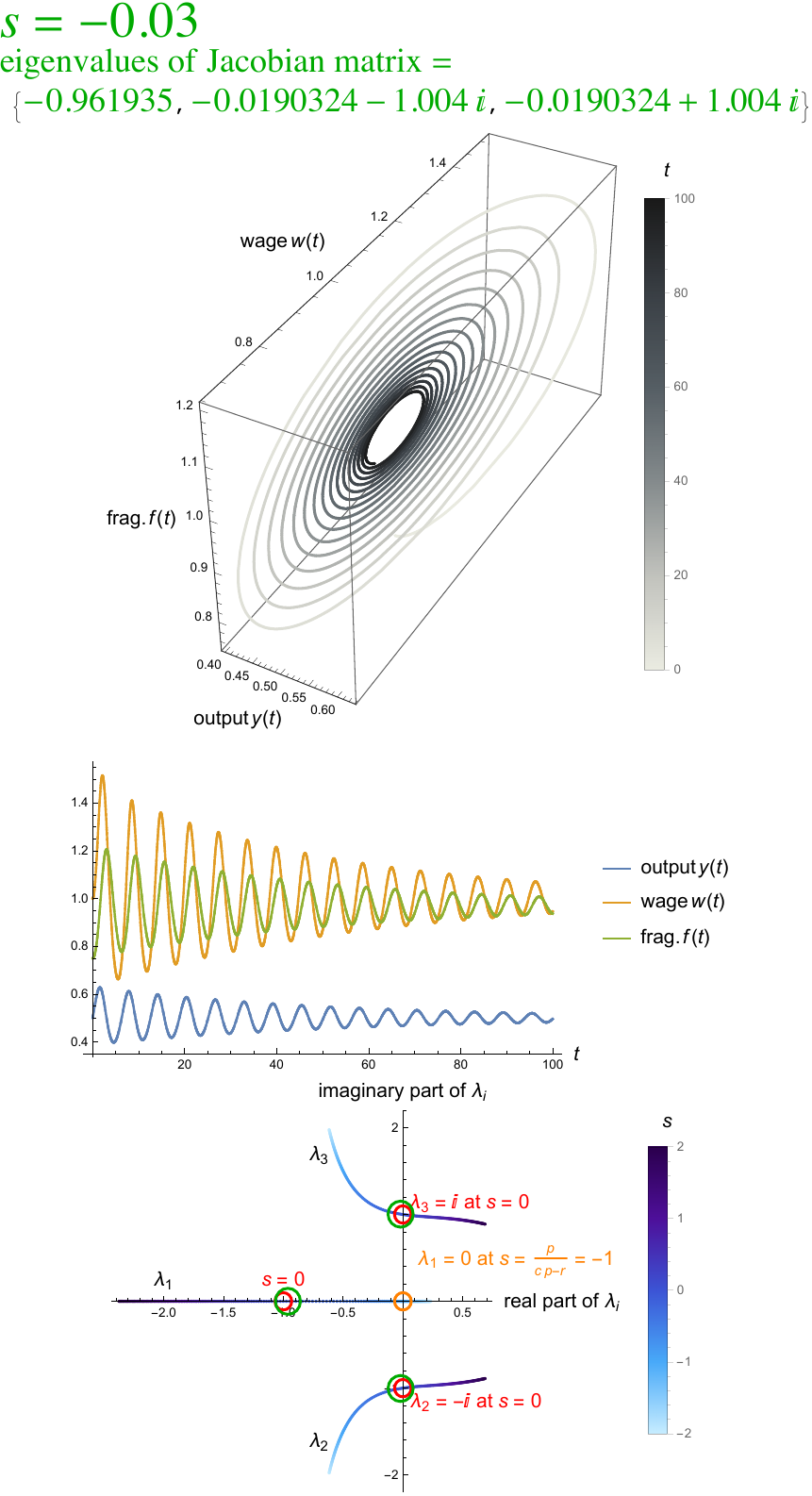}
\caption{Orbit for $s=-0.01$. 
Time increases from $0$ (light gray) to $200$ (dark gray). The dynamics spiral to a fixed point. 
Here, the parameters are the same as those used by~\citet[page 12]{Stockhammer2014}: $c=3/2, p = 2, r = 5,y(0) = 1/2, f(0) = 3/4, w(0) = 1$. 
}
\label{fig:orbit_s=-0p03}
\end{center}
\end{figure}

\clearpage


\end{document}